\documentclass[12pt]{amsart}

\usepackage{epic,amsfonts,color,epsfig,amssymb,latexsym,amsmath,hyperref}
\usepackage[super]{nth}

\newtheorem{thm}{Theorem}[section]
\newtheorem{hil}{Hilbert's \nth{17} Problem}

\newtheorem{que}{Question}[section]
\newtheorem{cor}[thm]{Corollary}
\newtheorem{lem}[thm]{Lemma}

\newcommand{\dia}{$\diamond$}

\newtheorem{ex}[thm]{Example}

\newtheorem{dfn}[thm]{Definition}

\newtheorem{rem}{Remark}[section]

\newcommand{\N}{\mathbb{N}}
\newcommand{\ox}{\bar{x}}
\newcommand{\ov}{\bar{v}}
\newcommand{\R}{\mathbb{R}}
\newcommand{\Z}{\mathbb{Z}}
\newcommand{\newt}{\mathrm{Newt}}
\newcommand{\conv}{\mathrm{Conv}}
\newcommand{\supp}{\mathrm{Supp}}
\newcommand{\pos}{\mathrm{Pos}}
\newcommand{\sq}{\mathrm{Sq}} 
\newcommand{\Rn}{\R^n}
\newcommand{\Zn}{\Z^n}

\newcommand{\Int}{\int\limits}

\newcommand{\abs}[1]{\left\lvert#1\right\rvert}
\newcommand{\vol}[1]{\mathrm{Vol}\!\left(#1\right)}
\newcommand{\norm}[1]{\left\lVert#1\right\rVert}

\author{Alperen A. Erg\"{u}r}
\email{erguer@math.tu-berlin.de}
\address{Technische Universit\"at Berlin, Institut f\"ur Mathematik, Sekretariat MA 3-2, Straße des 17. Juni 136, 10623, Berlin, Germany}
\thanks{Partially supported by NSF-MCS grant DMS-0915245, NSF-CAREER grant DMS-1151711, and Einstein Foundation, Berlin.}
\title[Multihomogenous nonnegative polynomials and sums of squares]{\mbox{} \\ 
\vspace{-1in} Multihomogenous nonnegative polynomials and sums of squares}

\begin{document}

\maketitle 

\vspace{-.4in} 
\begin{abstract} 
We refine and extend quantitative bounds, on the fraction of nonnegative 
polynomials that are sums of squares, to the multihomogenous case.
\end{abstract} 

\section{Introduction}
Let $\R\!\left[\ox\right]:=\R[x_1,\ldots,x_n]$ denote the ring of real 
$n$-variate polynomials and let $P_{n,2d}$ denote the vector space of forms (i.e homogenous polynomials) of degree $2d$ in 
$\R\!\left[\ox\right]$. 
A form $p \in P_{n,2d}$ is called {\em non-negative} if $p(\ox) \geq 0$ for every $\ox \in \Rn$.
The set of {\em non-negative} forms in $P_{n,2d}$ is closed under nonnegative linear combinations and 
thus forms a cone. We denote the cone of nonnegative degree $2d$ forms by 
$\pos_{n,2d}$. A fundamental problem 
in polynomial optimization and real algebraic geometry is to efficiently certify non-negativity for real forms, i.e., membership in $\pos_{n,2d}$. 

If a real form can be written as a sum of squares of other real forms then it is evidently non-negative. Polynomials in $P_{n,2d}$ that can be represented as sums of squares of real forms form a cone that we denote by $\sq_{n,2d}$. Clearly, $\sq_{n,2d}\subseteq\pos_{n,2d}$. We are then lead to the following question. 
\begin{que} \label{H1}
\label{q:often1}
For which pairs of  $(n,2d)$ do we have $\sq_{n,2d}=\pos_{n,2d}$?  
\end{que}

\noindent 
Hilbert showed that the answer to Question \ref{q:often1} is affirmative exactly for $(n,2d)\in (\{2\}\times 2\N) \cup (\N\times\{2\}) \cup \{3,4\}$ \cite{Hil88} . Hilbert's proof was not constructive: The first well known example of a non-negative form which is not sums of squares is due to Motzkin from around 1967: $x_3^6 + x_1^2x_2^2 (x_1^2+x_2^2-3x_3^2)$.

Hilbert included a variation of Question \ref{q:often1} in his famous list of  problems for \nth{20} century mathematicians: 
\begin{hil}
Do we have, for every $n$ and $2d$, that every $p\in \pos_{n,2d}$ is 
a sum of squares of {\em rational functions}? 
\end{hil}

\noindent 
Artin and Schreier solved Hilbert's \nth{17} Problem affirmatively around 1927  \cite{Artin}. However there is no known efficient and general algorithm for finding the asserted collection of rational functions for a given input $p$. 

Despite the computational hardness of finding a representation as a sum of squares of rational functions, obtaining a representation as a sum of squares of {\em polynomials} (when possible) can be done efficiently via {\em semidefinite programming} (see, e.g., \cite{parrilo}).  This connection to complexity theory strongly motivates a better understanding of the limits of semidefinite programing  approach to polynomial optimization. In this respect, we should note that for many problems of interest in algebraic geometry, forms with a special structure (e.g., sparse polynomials)  behave differently than generic forms of degree $2d$. So, we would like to study limits of sums of squares method in a more adaptive way that incorporates sparsity patterns. 

We first recall  the notion of Newton polytope and then a theorem of Reznick: For any  $p(\ox)=\sum_{\alpha \in \Zn} c_\alpha x^{\alpha}$ with $\alpha=(\alpha_1,\ldots,\alpha_n)$ and $\ox^{\alpha}=x^{\alpha_1}_1\cdots x^{\alpha_n}_n$,  the {\em Newton polytope} of $p$ is the convex hull $\newt(p):=\conv(\{\alpha\; | \; c_\alpha\neq 0\})$. 

\begin{thm} \cite[Thm.\ 1]{Rez78}
If $p=\sum^r_{i=1} g_i^2$ for some $g_1,\ldots,g_r\in\R\!\left[\ox\right]$ then $\newt(g_i) \subseteq \frac{1}{2} \newt(p)$ for all $i$.
\end{thm}

\noindent 
This theorem enables us to refine comparison of the cones of  {\em sums of squares} and {\em non-negative} polynomials to be more sensitive to monomial term structure. 
\begin{dfn}
For any polytope $Q\!\subset\!\Rn$ with vertices in $\Zn$, let 
$N_Q:=\#(Q\cap \Zn)$,\linebreak $c=(c_\alpha \; | \; \alpha\in Q\cap \Zn)$, $p_{c}(\ox)=\sum_{\alpha \in Q\cap 
\Zn} c_{\alpha} \ox^{\alpha}$ and then define\\  
\mbox{}\hfill 
$\pos_Q:=\{c\in\R^{N_Q}\; | \; p_{c}(\ox) \geq 0 \; \text{for every} \; x \in \mathbb{R}^n \}$\hfill\mbox{}\\  
\mbox{}\hfill \ \ \ $\sq_Q:=\{c\in\R^{N_Q}\; | \; 
 p_{c}(\ox)=\sum_{i} q_i(\ox)^2 \; \text{where} \; \newt(q_i) \subseteq \frac{1}{2}Q \}$\hfill\mbox{}\dia 
\end{dfn}

Now we can rewrite Question \ref{H1} in a more refined way.

\begin{que} \label{H2}
For which lattice polytopes $Q \subseteq \mathbb{R}^n$, we have $\pos_Q = \sq_Q$ ?
\end{que}

\noindent For the case of lattice polytopes $Q$ where $\frac{1}{2}Q$ is also a lattice polytope, Question \ref{H2} is solved by Blekherman, Smith and Velasco \cite{velasco}. Their work shows that for $Q=2P$, if $2P \neq P + P$, then $\pos_Q \neq \sq_Q$. Their work also provides a complete classification of the lattice polytopes $Q=2P=P+P$ where $\pos_Q = \sq_Q$ is achieved. 
  
\subsection{Quantitative Aspects of Hilbert's \nth{17} Problem}  

Suppose a Newton polytope $Q$ with $\pos_Q \neq \sq_Q$ is given, and one is interested in quantifying the gap between $\sq_Q$ and $\pos_Q$. For instance, for a given nonnegative polynomial $p \in \pos_Q$ how likely is it to have $p \in \sq_Q$? Greg Blekherman studied this problem in the case $Q=\Delta_{n,2d}$ where  

$$ \Delta_{n,2d} := \{ x \in \mathbb{Z}^n : x_i \geq 0 \; , \;  \sum_i x_i = 2d \} $$

\noindent and he concluded that if $d \geq 2$ is fixed and $n \rightarrow \infty$, for a given $p \in \pos_{\Delta_{n,2d}}$ we have $p \notin \sq_{\Delta_{n,2d}}$ almost surely \cite{Ble06}.

We extend quantitative comparison of the two cones $\pos_Q$ and $\sq_Q$ to the case of {\em multihomogenous} polynomials. We have tried to develop general methods than can allow study of arbitrary Newton polytopes, rather than developing ad hoc methods for multihomogenous case. However, there remains some technicalities to be resolved before we can handle arbitrary polytopes, mainly due to not being able to define the ``correct'' metric structure. So, we content ourselves with the multihomogenous case in this article. 

\begin{dfn}  
Assume henceforth that $n=n_1+\cdots+n_m$ 
and $d=d_1+\cdots+d_m$, with $d_i,n_i\!\in\!\N$ for all $i$, 
and set $N:=(n_1,\ldots,n_m)$ and $D:=(d_1,\ldots,d_m)$. We 
will partition the vector $\ox=(x_1,\ldots,x_n)$ into $m$ 
sub-vectors $\ox_1,\ldots,\ox_m$ so that $\ox_i$ 
consists of exactly $n_i$ variables for all $i$, and say that 
$p\in\R\!\left[\ox\right]$ is {\em $(N,D)$ homogenous} if 
and only if $p$ is\linebreak 
\scalebox{.93}[1]{homogenous of degree $d_i$ with respect to $\ox_i$ for 
all $i$. Finally, let $\Delta_{N,D}:=\Delta_{n_1,d_1}\times \cdots \times \Delta_{n_m,d_m}$. 
\dia} 
\end{dfn}
\begin{ex}
\scalebox{.93}[1]{$p(\ox):=x_1^3x_4^2+x_1x^2_2x_5^2+x^3_3x_4x_5$
is $(N,D)$ homogenous with $N=(3,2)$ and}\linebreak 
\scalebox{.92}[1]{$D=(3,2)$. (So $\ox_1=(x_1,x_2,x_3)$ and 
$\ox_2=(x_4,x_5)$.) In particular, 
$\newt(p)\subseteq \Delta_{N,D}=\Delta_{3,3}\times \Delta_{2,2}$. \dia}  
\end{ex}  

Multihomogenous forms appeared before in the works of several mathematicians. We refer the interested reader to \cite{clr} and references therein for the extensive history of nonnegative multiforms. The following theorem from \cite{clr} is the most relevant to our interest.

\begin{thm}\label{CLRmulti}(Choi, Lam, Reznick)
Let $N=(n_1,n_2,\ldots,n_m)$ and $D=(2d_1,2d_2,\ldots,2d_m)$ where $n_i \geq 2$ and $d_i \geq 1$ then $\pos_{\Delta_{N,D}}=\sq_{\Delta_{N,D}}$ if and only if $m=2$ and 
$(N,D)$ is either $(2,n_2;2d_1,2)$ or $(n_1,2;2,2d_2)$. 
\end{thm}

Our result can be viewed as a quantitative version of the theorem of Choi, Lam and Reznick. In order to state the main result we need to introduce the following function on subsets of $P_{N,D}$. 

\begin{dfn}
For a fixed partition, $(N,D)$ with $n=n_1+n_2+\ldots+n_m$, and $2d=2d_1+2d_2+ \ldots + 2d_m$. Let $S^{n_i-1}$ denote the standard unit $(n_i-1)$-sphere in $\R^{n_i}$, and let $\sigma_i$ be the uniform measure on $S^{n_i-1}$ with $\sigma_{i}(S^{n_i-1})=1$. We define  $S:=S^{n_1-1}\times \cdots \times S^{n_m-1}$, and let $\sigma := \sigma_1 \times \ldots \times \sigma_{m}$ be the product measure on $S$.

$$ P_{N,D}:=\{p \in \R\!\left[\ox\right] \text{ homogeneous of  type } (N,D)\} $$

\noindent and  
$$ C_{N,D}:=\left\{p \in P_{N,D} \; | \; \int_S p \: d\sigma=1\right\} $$ 

\noindent For any $X\!\subseteq\! P_{N,K}$  we set 

$$\mu(X)= \left( \frac{vol(X \cap C_{N,K})}{vol(B)} \right)^{\frac{1}{\dim(P_{N,D})}} $$

\noindent where $B$ is the unit ball with respect to the $L_2$ inner product introduced in the third section.
\end{dfn}

For the case of the polytope $\Delta_{N,D}=\Delta_{n_1,2d_1} \times \ldots \times \Delta_{n_m,2d_m}$, there is yet another family of non-negative polynomials that are easy to manipulate:
 
$$ L_{\Delta_{N,D}}:=\{ p \in \pos_{\Delta_{N,D}} : p=\sum_{i} l_{i1}^{2d_1} l_{i2}^{2d_2} \cdots l_{im}^{2d_m} \; \text{where} \; l_{ij} \; \text{are linear forms in} \; \ox_j \}  $$

\noindent Now we can state the main result of this paper.  

\begin{thm} \label{thm:main} 
Let $N=(n_1,n_2,\ldots,n_m)$ and $D=(2d_1,2d_2,\ldots,2d_m)$, then the following bounds hold.
  
$$  \frac{1}{4\sqrt{\max_i n_i}}  \prod_{i=1}^m (2d_i+1)^{-\frac{1}{2}} \leq  \mu(\pos_{\Delta_{N,D}}) \leq c_o $$
$$ c_1  \pi^{-2d} \prod_{i=1}^m c^{-\frac{d_i}{2}} \left( \frac{n_i}{2} + 2d_i\right)^{-\frac{d_i}{2}}  \leq  \mu(\sq_{\Delta_{N,D}}) \leq c_2 \prod_{i=1}^m (\frac{n_i+d_i}{cd_i})^{-\frac{d_i}{2}} $$
 $$ c_1 \pi^{-2d}  \prod_{i=1}^m (\frac{n_i}{2}+2d_i)^{-d_i} \leq  \mu(L_{\Delta_{N,D}}) \leq  4 \pi^{-2d}  \sqrt{\max_i n_i}  \prod_{i=1}^m (2d_i+1)^{\frac{1}{2}} (\frac{n_i}{2d_i})^{-d_i} $$
where $c_i$ are absolute constants with $1 \leq c_o \leq 5 $, $0 \leq c_1 \leq 1$, $1 \leq c_2 \leq 4$ and $c=2^{10} e$.
\end{thm} 

For the special case $D=(2,2,\ldots,2)$ we have the following corollary.

\begin{cor} \label{funQ}
Let $N=(n_1,n_2,\ldots,n_m)$ and $D=(2,2,\ldots,2)$, then the following bound holds.
$$  c_4  \prod_{i=1}^m \pi^{-2} c^{-\frac{1}{2}} (\frac{n_i+4}{2})^{-\frac{1}{2}} \leq \frac{\mu(\sq_{\Delta_{N,D}})}{\mu(\pos_{\Delta_{N,D}})}  \leq 4c_2 \sqrt{\max_i n_i} \prod_{i=1}^m (\frac{n_i+1}{3c})^{-\frac{1}{2}} $$
\noindent where $c_4=\frac{c_1}{5}$, $c$ and $c_2$ are absolute constants as in the main theorem.
\end{cor}

\noindent Corollary \ref{funQ} happens to have some interesting implications in quantum information theory. We have recently learned that  Klep, McCullough, Sivic and Zalar obtained similar bounds to Corollary \ref{funQ} with more explicit constants \cite{Klep}. We refer the interested reader to \cite{Klep} for an exposition of the connection to  quantum information theory, and for a very detailed analysis of the bounds in Corollary \ref{funQ}.

 To compare our result with Blekherman's bounds \cite{Ble06} let us  consider an even further special case of Corollary \ref{funQ}.
\begin{cor} \label{demonstrate}
Assume  $n=d.n_1$, and we have the following partition: $N=(n_1,n_1, \ldots, n_1)$ and $D=(2,2, \ldots, 2)$. Theorem \ref{thm:main}, gives the following bounds:

$$  c_1 \pi^{-2d} c^{-\frac{d}{2}} (2+\frac{n}{2d})^{-\frac{d}{2}} \leq  \frac{ \mu\!\left(\sq_{\Delta_{N,D}} \right)}{\mu\!\left(\pos_{Q_{N,D}}\right)} \leq  c_2 (\frac{n}{cd})^{\frac{-d+1}{2}}   $$ 
where $c$, $c_1$ and $c_2$ are absolute constants.
\end{cor}

\noindent 
Note that the Newton polytope considered in the case above is contained in $\Delta_{n,2d}$. In particular, Blekherman's Theorems 4.1 and 6.1 from \cite{Ble06}  give the following estimates: 
$$  \frac{n^{\frac{d+1}{2}}}{(\frac{n}{2}+2d)^d} \frac{c_1 d! (d-1)!}{4^{2d} (2d)!}  \leq \frac{\mu\!\left(\sq_{\Delta_{n,2d}}\right)} {\mu\!\left(\pos_{\Delta_{n,2d}}\right)} \leq  \frac{ c_2 4^{2d} (2d)! \sqrt{d} }{d!}  n^{\frac{-d+1}{2}} $$
where $c_1$ and $c_2$ are absolute constants. 

Bounds in Corollary \ref{demonstrate} depends on $\frac{n}{d}$ instead of $n$ which shows the effect of underlying multihomogeneity. In particular in the cases that $d$ and $n$ are comparable our bounds behave significantly differently then the bounds of Blekherman.  

As a high level summary, Theorem \ref{thm:main} proves that if we assume multihomogeneity on the set of variables $\ox_1,\ox_2,\ldots,\ox_m$, bounds derived in Blekherman's work for the ratio of sums of squares to non-negative polynomials, repeats itself for every set of variable $\ox_j$. 

The rest of the article is structured as follows; we first review background material from convex geometry and analytic theory of polynomials. Then, we introduce two inner products on multihomogenous forms and study relations between them. These two inner products introduce two different notions of duality, which turns out to be both very useful. After these preliminary sections, we have three sections that provide bounds for $\mu(\pos_{\Delta_{N,D}})$, $\mu(\sq_{\Delta_{N,D}})$ and $\mu(L_{\Delta_{N,D}})$ respectively. 

\section{Background Material}

\subsection{Convex Geometric Analysis}

We begin with recalling a theorem of Fritz John \cite{john}.

\begin{thm}(John's Theorem) \label{john}
Every convex body $K \subset \mathbb{R}^n$ is contained in a unique ellipsoid of the minimal volume $E_{min}$.  Moreover, 

$$ \frac{1}{n} E_{min}  \subset K \subset E_{min} .$$

\noindent The minimal volume ellipsoid $E_{min}$ is the Euclidean unit ball $B_2^{n}$ if and only if the following conditions are satisfied:  $K \subseteq B_2^n$, there are unit vectors $(u_i)_{i=1}^m$ on the boundary of $K$ and positive real numbers $c_i$ such that

$$ \sum_{i=1}^m c_i u_i=0$$

\noindent and for all $x \in \mathbb{R}^n$ we have
$$ \sum_i c_i \langle u_i , x \rangle^2 = \norm{x}_2^2$$

\noindent For the convex bodies $K$ with $E_{min}=B_2^{n}$, we say $K$ is John's position.
\end{thm}

The criterion in Theorem \ref{john} for John's position is called John's decomposition of identity. One way to view John's decomposition is to observe that the family of unit vectors $\{ u_i \}_{i=1}^{m}$ work like an  orthogonal basis in $\mathbb{R}^{n}$. Painless non-orthogonal decompositions like John's decomposition has long been studied in applied mathematics. The lemma below is a standard fact from frame theory, and is included here for completeness.

\begin{lem} \label{frame}
We denote the map that sends $x$ to $\langle  x , y \rangle z$ by $ y \otimes z$. Then the following are equivalent

\begin{enumerate}
\item   
$$ \mathbb{I}=\sum_{i} c_i u_i \otimes u_i $$
\item  For every $x \in \mathbb{R}^n$
$$ x=\sum_{i} c_i \langle x,u_i \rangle u_i $$
\item For every $x \in \mathbb{R}^n$ 
$$ \sum_i c_i \langle u_i , x \rangle^2 = \norm{x}_2^2$$

\end{enumerate}
\end{lem}

Another perspective on John's decomposition is to view the decomposition as a discrete measure supported on the vectors $u_i$ with weights $c_i$, and the identity being the covariance matrix of the measure. This measure theoretic interpretation is formalized in the notion of isotropic measures which we present below. 

\begin{dfn}
A finite Borel measure $Z$ on the sphere $S^{n-1}$ of a $n$ dimensional real vector space $V$ is  said to be isotropic if 

$$   \norm{x}_2^2 = \int_{S^{n-1}}  \langle x ,u \rangle^2 Z(u)  $$

\noindent for all $x  \in V$. Moreover, we define the centroid of a measure $Z$ supported on the sphere $S^{n-1}$ as 

$$ \frac{1}{Z(S^{n-1})} \Int_{S^{n-1}} u \; Z(u)  .$$

\noindent We say the measure is centered at $0$ if the centroid is the origin.
\end{dfn}

An isotropic measure supported on the sphere with centroid $0$ is the continuous analog of John's decomposition. It is also known that a convex body is in John's position if and only if the touching points of the convex body to the unit ball supports an isotropic measure with centroid at the origin \cite{petros}. 

Suppose that a convex body $K$ is given where $E_{min}$ is $B_2^n$. That is, the touching points of $K$ to $B_2^n$ is the support of an isotropic measure. Now, let $\widetilde{K}$ be the convex hull of these touching points. By John's Theorem, $K$ and $\widetilde{K}$ have the same minimal volume ellipsoid. However, the volume of $K$ and the volume of $\widetilde{K}$ can differ up to $n$. 

We observe in this article that for an interesting convex body $K$ which is dual to the cone of nonnegative polynomials, we have $\widetilde{K}= K$. It seems that for this special case, we need refined estimates instead of using John's Theorem. Thankfully such improved estimates has already been worked out by Lutwak, Yhang and Zhang.  We need to introduce one more definition to state the their result. For a convex body $K \subseteq \mathbb{R}^n$ the polar of $K$ denoted by $K^{\circ}$ is defined as follows:

$$ K^{\circ}:= \{ x \in \mathbb{R}^n : \langle x , y \rangle \leq 1 \; \text{for all} \; y \in K \} $$

\begin{thm} [Lutwak, Yhang, Zhang] \label{LYZ} \cite{LYZ07} 
If $Z$ is an isotropic measure on $S^{n-1}$ whose centroid is at the origin and $Z_{\infty}=\conv(\supp(Z))$, then we have 
$$\abs{Z_{\infty}^{\circ}} \leq \frac{n^{\frac{n}{2}} (n+1)^{\frac{n+1}{2}}}{n!} $$  

\noindent where $\abs{Z_{\infty}^{\circ}}$ denotes the volume of $Z_{\infty}^{\circ}$.
\end{thm}

We will need few other results from classical and modern convexity. Here is a bit of terminology: For a convex body $K \subseteq \mathbb{R}^n$, the support function $h_K$ is given by 

$$ h_K(u)=\max_{x \in K} \langle x, u \rangle  $$

\noindent The difference $\omega_K(u)=h_K(u)+h_K(-u)$ is the width of $K$ in direction $u$. The mean width of $K$ is defined as follows. 

$$ \omega(K) := \Int_{S^{n-1}} \frac{\omega_K(u)}{2} \:  \sigma(u) = \Int_{S^{n-1}} h_{K}(u) \: \sigma(u) $$ 

\noindent Now, we can state Urysohn's inequality  \cite{petros}.

\begin{thm} [Urysohn] \label{urysohn} For every convex body $K$ in $\mathbb{R}^n$, we have
$$ \left( \frac{\abs{K}}{\abs{B_2^n}} \right)^{\frac{1}{n}}  \leq \omega(K) $$
\end{thm}

\noindent A very much related quantity to the support function of convex body $K$ is the Gauge function $G_K$ of $K$.

$$ G_K(u) := inf \{ \lambda : \lambda u  \in K \}  $$

\begin{lem} \label{standard} \cite{pisier} For a convex body $K \subseteq \mathbb{R}^n$ with $0 \in K$, the ratio of the volume of $K$ to the unit ball $B_2^n$ can be expressed as follows.
$$  \left( \frac{\abs{K}}{\abs{B_2^n}} \right)^{\frac{1}{n}} = \left( \int_{S^{n-1}} \abs{G_K(u)}^{-n}  \; \sigma(u) \right)^{\frac{1}{n}} $$
\end{lem}

\noindent We have two more theorems to present in this section: the Santalo inequality, and the reverse Santalo inequality of Bourgain and Milman \cite{BouMil87}.

\begin{thm} [Bourgain-Milman] \label{bigboe}
There exists an absolute constant $c_s >0$ such that for any convex body $K \subseteq \mathbb{R}^n$ that includes the origin, we have

$$ c_s \leq \left( \frac{\abs{K}}{\abs{B_2^n}} \right)^{\frac{1}{n}} \left( \frac{\abs{K^{\circ}}}{ \abs{B_2^n}} \right)^{\frac{1}{n}} $$ 

\end{thm}

We must note that Bourgain-Milman paper \cite{BouMil87} proves Theorem \ref{bigboe} above for symmetric convex bodies. The conclusion for all convex bodies $K$ is a corollary of Bourgain-Milman theorem applied to difference body of $K$ together with the usage Rogers-Shephard inequality. Throughout the article we will keep this constant as $c_s$ referring to reverse Santalo inequality.

\begin{thm}[Santalo Inequality] \label{Santalo}
For every convex body $K \subseteq \mathbb{R}^n $ there exists a unique point $z$ in the interior of $K$, with the following extremal property:

$$ \abs{(K-z)^{\circ}} = \min_{x \in K}  \abs{ (K-x)^{\circ} }  .$$

\noindent This unique point is called the Santalo point of $K$, and for this particular point $z$ we have the following inequality.

$$ (\frac{\abs{K}}{\abs{B_2^n}})^{\frac{1}{n}} (\frac{\abs{(K-z)^{\circ}}}{\abs{B_2^n}})^{\frac{1}{n}}  \leq 1 $$

\noindent where $B_2^n$ is the Euclidean unit ball of $\mathbb{R}^n$.
\end{thm}

\subsection{Harmonic Polynomial Basics}

In this section we will present some basic results from analytic theory of polynomials. These results can be found in any textbook on the subject, we suggest \cite{sh} for a down to earth presentation with elementary proofs. We start by defining the Laplace operator. 

$$ \mathcal{L}: P_{n,2d} \rightarrow \mathbb{R}  \; \; , \; \; \mathcal{L}(f) = \frac{\partial^{2} f}{\partial x_1^2} + \frac{\partial^{2} f}{\partial x_2^2} + \ldots + \frac{\partial^{2} f}{\partial x_n^2} $$

\noindent We say a polynomial  $f \in P_{n,2d}$ is harmonic if $\mathcal{L}(f)=0$, and denote the vector space of harmonic degree $2d$ forms with $H_{n,2d}$. The following is well-known: for every $p \in P_{n,2d}$, there exists a unique decomposition $p=\sum_{i=1}^d (x_1^2+x_2^2+\ldots+x_n^2)^{i} h_i$ where $h_i \in H_{n,2d-2i}$. We will explain this decomposition in the context of multihomogenous forms in the next section. 

For a given point $v \in S^{n-1}$, we consider the following pointwise evaluation map on $H_{n,2d}$.

$$ l_v : H_{n,2d} \rightarrow \mathbb{R} \; \;, \;  \;  l_v(h)=h(v) $$

\noindent It is well-known that for every point $v \in S^{n-1}$, there exist a polynomial $q_v \in H_{n,2d}$ called zonal harmonic, that satisfies the following equality for all $h \in H_{n,2d}$.

$$ l_v(h)= h(v) = \int_{S^{n-1}} h(x) q_v(x) \; \sigma(x) $$ 

\noindent where $\sigma$ is the uniform measure on $S^{n-1}$ with $\sigma(S^{n-1})=1$. Zonal harmonics have quite nice properties as we summarize below. 
\begin{lem} \label{b1}

\begin{enumerate}
\item For all $T \in O(n)$ with $T(v)=v$ and for all $x \in S^{n-1}$,  we have

 $$ q_v(T(x))=q_v(x) $$ 

\noindent Conversely, if a form $h \in H_{n,2d}$ satisfies $h(Tx)=h(x)$ for all $T \in O(n)$ with $T(v)=v$, then $h=c q_v$ for a constant $c$.

\item There exists a univariate polynomial $P_{n,2d}$ called ultraspherical polynomial, with the following property:

$$  q_v(x) = q_v(v) P_{n,2d}( \langle x , v \rangle )  $$

\noindent for all $v \in S^{n-1}$.

\item The ultraspherical polynomial $P_{n,2d}$ satisfies Rodrigues' formula; that for all $f \in C^{(n)}[-1,1]$, we have

$$ \int_{-1}^1 f(t) P_{n,2d}(t) (1-t^2)^{\frac{n-3}{2}} \; \mathrm{d} t = \frac{\Gamma(\frac{n-1}{2})}{2^{2d} \Gamma(2d+\frac{n-1}{2})}  \int_{-1}^{1} f^{(2d)}(t)(1-t^{2})^{2d+\frac{n-3}{2}} \; \mathrm{d} t$$

\item The following holds for all $h \in H_{n,2d}$, and for all $x \in S^{n-1}$:

$$  h(x) = \int_{S^{n-1}} h(v) q_v(x) \; \sigma(v)  $$
 
\noindent where $\sigma$ is the uniform measure on the sphere with $\sigma(S^{n-1})=1$. \dia
\end{enumerate}  
\end{lem}

There are several immediate consequences of  Lemma \ref{b1}; combining second and fourth item one can deduce that $q_w(v)=q_v(w)$ and that $q_v(v)=q_w(w)$ for all $v,w \in S^{n-1}$. This in turn implies $q_v(v)=\dim(H_{n,2d})$, and $P_{n,2d}(1)=1$. Now, we present a result attributed to Hecke and Funk. 

\begin{thm}[Hecke-Funk Formula]
Let $K$ be a measurable function on $[-1,1]$ where the integral

$$ \int_{-1}^1 \abs{K(t)} (1-t^2)^{\frac{n-3}{2}} \; \mathrm{d}t  $$

\noindent is finite. Then, for all $h \in H_{n,2d}$ and any $x \in S^{n-1}$, we have

$$ \int_{S^{n-1}} K(\langle x , v \rangle )  h(v) \; \sigma(v) = \left( \frac{\abs{S^{n-2}}}{\abs{S^{n-1}}} \int_{-1}^1 K(t) P_{n,2d}(t) (1-t^2)^{\frac{n-3}{2}} \; \mathrm{d}t \right) h(x) $$

\noindent where $\sigma$ is the uniform measure on $S^{n-1}$ with $\sigma(S^{n-1})=1$, and $P_{n,2d}$ is the ultraspherical polynomial as in the Lemma \ref{b1}.

\end{thm}

\begin{proof}[Sketch of the Proof]
If $h=q_w$ for some $w \in S^{n-1}$, then we have

$$ F(x) = \int_{S^{n-1}} K(\langle x , v \rangle )  q_w(v) \; \sigma(v) = q_w(w) \int_{S^{n-1}}  K(\langle x , v \rangle )  P_{n,2d}(\langle w , v \rangle) \; \sigma(v)  $$

From this expression it is clear that for any $T \in SO(n)$ with $T(w)=w$, we have $F(Tx)=F(x)$. By first item in Lemma \ref{b1}, we have $ F(x) = c q_w(x) $. Using the special case $w=x$, gives 

$$c = \frac{\abs{S^{n-2}}}{\abs{S^{n-1}}} \int_{-1}^1 K(t) P_{n,2d}(t) (1-t^2)^{\frac{n-3}{2}} \; \mathrm{d}t  $$

\noindent Now, for any $h \in H_{n,2d}$ we have

$$ \int_{S^{n-1}} K(\langle x , v \rangle )  h(v) \; \sigma(v)  = \int_{S^{n-1}} \int_{S^{n-1}} K(\langle x , v\rangle)  h(w) q_v(w) \; \sigma(w) \; \sigma(v)   $$ 

Applying Fubini's Theorem, and using $q_v(w)=q_w(v)$ completes the proof.
\end{proof}

We will use the following two identities in the coming proof.  

$$ \abs{S^{n-1}} \int_{S^{n-1}} x_i^{2d} \; \sigma(x)=\frac{2 \Gamma(d+\frac{1}{2}) }{ \Gamma(d+\frac{n}{2})}  $$

$$ \sqrt{\pi} \Gamma(x) = 2^{x-1} \Gamma(\frac{x+1}{2}) \Gamma(\frac{x}{2}) $$

The first identity can be found in \cite{Bar02}, the second is folklore. In what follows we apply Hecke-Funk formula to $K(t)=t^{2d}$ and a harmonic polynomial $h$ with degree $2k$ for $k \leq d$.

\begin{lem} \label{pr}
Let $h \in H_{n,2k}$, then we have

$$ \int_{S^{n-1}} \langle v , w \rangle^{2d} h(v) \; \sigma(v) =  A_1 \pi^{-2k} \left( \frac{ d! \Gamma(d+ \frac{n}{2}) }{ (d-k)! \Gamma(d+k+\frac{n}{2})}  \right) h(w) $$

\noindent where $A_1 = \int_{S^{n-1}} x_i^{2d} \; \sigma(x)$.
\end{lem}

\begin{proof}

By Hecke-Funk Theorem, we have

$$ \int_{S^{n-1}} \langle v , w \rangle^{2d} h(v) \; \sigma(v)  =  \left( \frac{\abs{S^{n-2}}}{\abs{S^{n-1}}}  \int_{-1}^1 t^{2d} P_{n,2k}(t) (1-t^2)^{\frac{n-3}{2}} \; \mathrm{d}t \right) h(w)$$

By Rodrigues' formula, we have

$$ \int_{-1}^1 t^{2d} P_{n,2k}(t) (1-t^2)^{\frac{n-3}{2}} \; \mathrm{d}t  = \frac{(2d)! \Gamma(\frac{n-1}{2})}{2^{2k} (2d-2k)! \Gamma(2k+\frac{n-1}{2})}  \int_{-1}^{1} t^{2d-2k}(1-t^{2})^{2k+\frac{n-3}{2}} \; \mathrm{d} t $$

The right hand side of the integral can be interpreted as integrating $x_n^{2d-2k}$ over $4k+n-1$ sphere:

$$ \abs{S^{4k+n-1}} \int_{S^{4k+n-1}} x_n^{2d-2k} \; \sigma(x)  =  \abs{S^{4k+n-2}} \int_{-1}^{1} t^{2d-2k} (1-t^{2})^{2k+\frac{n-3}{2}} \; \mathrm{d} t  $$

If we plug-in these equations starting from $\int_{S^{n-1}} \langle v , w \rangle^{2d} h(v) \; \sigma(v) $, we have

$$ = \frac{\abs{S^{n-2}}}{\abs{S^{n-1}}} \frac{(2d)! \Gamma(\frac{n-1}{2})}{2^{2k} (2d-2k)! \Gamma(2k+\frac{n-1}{2})} \left( \frac{2 \Gamma(d-k+\frac{1}{2}) }{ \abs{S^{4k+n-2}} \Gamma(d+k+\frac{n}{2})}  \right) $$

We use  $\abs{S^{n-1}} \Gamma(\frac{n}{2}) =2\pi^{\frac{n}{2}}$. 

$$ = \frac{(2d)!}{2^{2k} (2d-2k)!  \pi^{2k+\frac{n}{2}} } \left(  \frac{ \Gamma(\frac{n}{2}) \Gamma(d-k+\frac{1}{2}) }{\Gamma(d+k+\frac{n}{2})}  \right)  $$ 

We use the second identity  $\sqrt{\pi} \Gamma(2d-2k+1) =2^{2d-2k} \Gamma(d-k+1) \Gamma(d-k+\frac{1}{2})$.

$$ = \frac{(2d)!}{2^{2d} \pi^{2k+\frac{n-1}{2}} } \left( \frac{\Gamma(\frac{n}{2}) }{ \Gamma(d-k+1) \Gamma(d+k+\frac{n}{2})}  \right)  = \frac{d!}{\pi^{2k+\frac{n}{2}}} \left( \frac{ \Gamma(d+\frac{1}{2}) \Gamma(\frac{n}{2}) }{ \Gamma(d-k+1) \Gamma(d+k+\frac{n}{2})}  \right) $$ 

Where we also used $\sqrt{\pi} (2d)! = 2^{2d} \Gamma(d+1) \Gamma(d+\frac{1}{2})$. Now we use the first identity for $A_1$.

$$A_1= \int_{S^{n-1}} x_i^{2d} \; \sigma(x) = \frac{2 \Gamma(d+\frac{1}{2}) }{ \abs{S^{n-1}} \Gamma(d+\frac{n}{2})} = \frac{\Gamma(\frac{n}{2}) \Gamma(d+\frac{1}{2}) }{ \pi^{\frac{n}{2}} \Gamma(d+\frac{n}{2})} $$

\noindent This completes the proof.
\end{proof}

\section{A Tale of Two Inner Products on Multihomogenous Forms}

For a fixed partition $N=(n_1,n_2,\ldots,n_m)$ with $n=n_1+n_2+\ldots+n_m$, and $D=(2d_1,2d_2, \ldots, 2d_m)$ with $2d=2d_1+2d_2+ \ldots + 2d_m$, we have defined $P_{N,D}$ to be the vector space of $n$-variate degree $2d$ forms that are $(N,D)$ homogenous. In this section we will introduce two inner products on $P_{N,D}$, and compare the geometry introduced by these two inner products. Let us also recall that we defined $S:=S^{n_1-1}\times \cdots \times S^{n_m-1}$. In the rest of the article, we let $\sigma_i$ to be the uniform measure on $S^{n_i-1}$ with $\sigma_i(S^{n_i-1})=1$, and let  $ \sigma = \sigma_1 \times \sigma_2 \times \ldots \times \sigma_m $  be the product measure on $S$. 

We will naturally consider the action of $O(n_1) \times O(n_2) \times \ldots O(n_m)$ on $P_{N,D}$. So, in short we denote this group with $O(N)$. Similarly we denote $SO(n_1) \times SO(n_2) \times \ldots \times SO(n_m)$ with $SO(N)$. For an element $U \in O(N)$, and $f \in P_{N,D}$ the action of $U$ on $f$ is defined by  $  U \circ f (x) := f (U^{-1} x)  $.

\begin{dfn}[Two Inner Products]
For $f,g \in P_{N,D}$, we define $L_2$ inner product as 

$$ \langle f,g \rangle := \Int_{S}{f(v)g(v)} \: \sigma(v)  $$

\noindent For $f(x)= \sum_\alpha c_\alpha x^{\alpha} \in P_{N,D}$ with 
$\alpha=(\alpha_1,\ldots,\alpha_n)$ we define the linear differential operator 
$$ D[f]:=\sum_\alpha c_\alpha \left(\frac{\partial^{\alpha_1}}{\partial x_1^{\alpha_1}} \cdots \frac{\partial^{\alpha_n}}{\partial x_n^{\alpha_n}} \right) $$
\noindent and set 

$$ \langle f,g \rangle_D:=D[f](g) $$

\noindent This way of defining $\langle f,g \rangle_D$,  introduces an inner product which we call the  ``differential'' inner product. \dia 
\end{dfn} 

We will list below some basic properties of differential inner product. First, for all $v \in S$ we define a corresponding useful form $\delta_v \in P_{N,D}$ as follows:

$$ \delta_v(x) := \langle \ov_1 , \ox_1 \rangle^{2d_1} \langle \ov_2 , \ox_2 \rangle^{2d_2} \ldots \langle \ov_m, \ox_m \rangle^{2d_m} $$

\begin{lem} \label{differential}
\begin{enumerate}
\item For all $p \in P_{N,D}$,  we have

$$\langle p , \delta_v \rangle_D = (2d_1)!(2d_2)! \ldots (2d_m)! p(v) $$

\item For all $p \in P_{N,D}$, and all $n$-variate forms $g,h$ with $gh \in P_{N,D}$, we have

$$ \langle p , gh \rangle_D = \langle D[g](p) , h \rangle_D = \langle D[h](p) , g \rangle_D  $$

\end{enumerate}
\end{lem}

\noindent Now, we define an operator $T$ which captures the relation between the differential and the $L_2$ inner products. The analog of this operator on $P_{n,2d}$ is attributed to Reznick \cite{Rez}.  

\begin{dfn}[T-operator]
$$ T:P_{N,D} \rightarrow P_{N,D}  \; \; , \; \;  T(f)= A^{-1} \int_S f(v) \delta_ v \; \sigma(v)   $$

\noindent $A=A_1 A_2 \ldots A_m$, and $A_i = \int_{S^{n_i-1}} \langle \ov_i , \ox_i \rangle^{2d_i} \; \sigma_i(x) $ for any vector $v=(\ov_1,\ldots, \ov_m) \in S$. \dia
\end{dfn}
The rationale for the constant $A$ is the following: Let $r_1 = x_1^2 + x_2^2 + \ldots + x_{n_1}^2 $, and let $r_i$ for $1 \leq i \leq m$ be defined similarly. Also let $r=r_1^{d_1} r_2^{d_2} r_3^{d_3} \ldots r_m^{d_m}$. We observe two things: $r(v)=1$ for all $v \in S$, and $r$ is fixed under the action of $O(N)$. So, it would be nice to have $T(r)=r$, which is equivalent to have $T(r)(x)=1$ for all $x \in S$. Let $x \in S$, then
 $$  T(r)(x) = A^{-1} \int_{S} \delta_v(x) \; \sigma(v) = A^{-1} \prod_{i=1}^m \int_{S^{n_i-1}} \langle \ov_i , \ox_i \rangle^{2d_i} \; \sigma_i(x) = 1  $$

\noindent Also note that $A_i$ can be written explicitly as we explained in the previous section. 

\begin{lem} \label{switch}
For all $f,g \in P_{N,D}$, we have  $ \langle T(f) , g \rangle_D = A^{-1} \prod_{i=1}^m (2d_i)!  \langle f , g \rangle  $
\end{lem}

\begin{proof}
$$  \langle T(f) , g \rangle_D = A^{-1} \int_S \langle f(v) \delta_v , g \rangle_D \; \sigma(v) = A^{-1} \prod_{i=1}^m (2d_i)! \int_S f(v) g(v) \; \sigma(v) $$
\end{proof}

We introduce harmonic polynomials in multihomogenous setup.

\begin{dfn}[N-Harmonic Polynomials]
For a partition $N=(n_1,n_2,\ldots,n_m)$ with $n=n_1+n_2+\ldots+n_m$, we define $\mathcal{L}_i$ to be the Laplace operator in variables $\ox_i$. For instance, 
$$ \mathcal{L}_1(p) =   \frac{\partial^{2} f}{\partial x_1^2} + \frac{\partial^{2} f}{\partial x_2^2} + \ldots + \frac{\partial^{2} f}{\partial x_{n_1}^2} $$

\noindent Then, for $p \in P_{N,D}$ we say $p$ is $N$-harmonic if 

$$ \mathcal{L}_1(p)=\mathcal{L}_2(p)= \ldots = \mathcal{L}_m(p)=0 $$ \dia
\end{dfn}

The operators $\mathcal{L}_i$ introduce an order on lattice points. We define the set of lattice points that are dominated by the vector $D=(2d_1,2d_2,\ldots,2d_m)$ as $\mathcal{I}(D)$. 

$$ \mathcal{I}(D) := \{ \alpha \in \mathbb{Z}^m : (-1)^{\alpha_i} = 1  \; \text{and}  \; 0 \leq \alpha_i \leq 2d_i  \; \text{for all} \; 1 \leq i \leq m \}  $$ 

\noindent We define $H_{N,\alpha}$ to be the vector space of $N$-harmonic polynomials in $P_{N,\alpha}$. Then, we have the following orthogonal decomposition result.

\begin{lem} \label{D-ort}
Let $\oplus^D$ denote the orthogonal decomposition with respect to differential inner product. $P_{N,D}$ can be decomposed into spaces of $N$-harmonic polynomials as follows:

$$ P_{N,D} = \oplus^D_{\alpha \in \mathcal{I}(D)} r_1^{d_1-\frac{\alpha_1}{2}} r_2^{d_2-\frac{\alpha_2}{2}} \ldots r_m^{d_m-\frac{\alpha_m}{2}}  H_{N,\alpha}$$
\end{lem}

\begin{proof}
Let $\alpha, \beta \in \mathcal{I}(D)$ with $\alpha \neq \beta$, and w.l.o.g. assume $\alpha_1 > \beta_1$. Now suppose $f \in r^{D-\alpha} H_{N,\alpha}$ and $g \in r^{D-\beta} H_{N,\alpha}$ where we used $r^{D-\alpha}$ for $r_1^{d_1-\frac{\alpha_1}{2}} r_2^{d_2-\frac{\alpha_2}{2}} \ldots r_m^{d_m-\frac{\alpha_m}{2}}$. Let 
$f=r^{D-\alpha} h_{\alpha}$ and $g=r^{D-\beta} h_{\beta}$. 

$$ \langle r^{D-\alpha} h_{\alpha} , r^{D-\beta} h_{\beta} \rangle_D = \langle D(r_1^{d_1-\frac{\beta_1}{2}})(r^{D-\alpha} h_{\alpha}) , r_2^{d_2-\frac{\beta_2}{2}} \ldots r_m^{d_m-\frac{\beta_m}{2}} h_{\beta} \rangle_D $$

\noindent Since we assumed $d_1 - \frac{\beta_1}{2} > d_1 - \frac{\alpha_1}{2} $, and we also assumed $\mathcal{L}_1(h_{\alpha})=0$, this yields $\langle f , g \rangle_D = 0$. That is $H_{N,\alpha}$ is orthogonal to $H_{N,\beta}$ for $\alpha \neq \beta$. 

Now let $E = \oplus^D_{\alpha \in \mathcal{I}(D)}  r^{D-\alpha} H_{N,\alpha}$, and assume that $E \neq P_{N,D}$. Then there exists $f \in P_{N,D}$ such that $f \bot E$ w.r.t to differential inner product. By assumption $f$ is not $N$-harmonic. W.l.o.g. say $\mathcal{L}_1(f) \neq 0$, and let $f_1=\mathcal{L}_1(f)$. If $f_1$ is $N$-harmonic, then we have

$$ \langle f , r_1^2 f_1 \rangle_D = \langle D[r_1^2](f) , f_1 \rangle_D = \langle \mathcal{L}_1(f) , f_1 \rangle_D = \langle f_1 , f_1 \rangle_D \neq 0 $$

\noindent This gives a contradiction since $r_1^2f_1 \in E$ and $f \bot E$. Assume $f_1$ is not $N$-harmonic, w.l.o.g. say $\mathcal{L}_m(f_1) \neq 0$ and say $\mathcal{L}_m(f_1) = f_2$. If $f_2$ is $N$-harmonic, we can play the same game with $r_1^2r_m^2 f_2$ and arrive to a contradiction. So, assume $f_2$ is not $N$-harmonic. This inductive reasoning will arrive to a contradiction eventually since all forms of degree $0$ are $N$-harmonic! 
\end{proof}

\begin{lem} \label{T}
Suppose $f \in P_{N,D}$ is given, and let $f_{\alpha}$ denote the projection of $f$ on $r^{D-\alpha}H_{N,\alpha}$. Then, we have

$$ T(f) = \sum_{\alpha \in \mathcal{I}(D)} \left( \prod_{i}^m  \pi^{-\alpha_i} \frac{d_i! \Gamma(d_i+\frac{n_i}{2})}{(d_i-\frac{\alpha_i}{2})! \Gamma(d_i+\frac{\alpha_i+n_i}{2})} \right) f_{\alpha} $$
\end{lem}

\begin{proof}
Just repeat Lemma \ref{pr} in every set variables $n_i$ of the partition $N$. 
\end{proof}

A consequence of Lemma \ref{T} is that for a given $h \in r^{D-\alpha} H_{N,\alpha}$ for some $\alpha \in \mathcal{I}(D)$, we have $T(h) \in r^{D-\alpha} H_{N,\alpha}$. For $\beta \in \mathcal{I}(D)$ with $\alpha \neq \beta$ and $g \in r^{D-\beta} H_{N,\beta}$ we then know from Lemma \ref{D-ort} that $\langle T(h) , g \rangle_D=0$. Combining Lemma \ref{switch} with Lemma \ref{D-ort}, we have 

$$  \langle T(h) , g \rangle_D = 0 = A^{-1} (2d_1)! (2d_2)! \ldots (2d_m)! \langle h ,g \rangle  $$
 
\noindent which shows that the vector spaces $r^{D-\alpha} H_{N,\alpha}$ and $r^{D-\beta} H_{N,\beta}$ are orthogonal to each other with respect to $L_2$ inner product as well. Hence, the decomposition in Lemma \ref{D-ort} and the projection in Lemma \ref{T} are also valid in $L_2$ inner product.   

Lemma \ref{T}, and the discussion in preceding paragraph shows that the vector spaces $r^{D-\alpha} H_{N,\alpha}$  are the eigenspaces of the operator $T$ with the eigenvalues being 

$$ \prod_{i}^m  \pi^{-\alpha_i} \frac{d_i! \Gamma(d_i+\frac{n_i}{2})}{(d_i-\frac{\alpha_i}{2})! \Gamma(d_i+\frac{\alpha_i+n_i}{2})} $$

\noindent This description of the eigenvalues allows us to bound the determinant of the map $T$ as follows. 

\begin{lem} \label{detT}

$$  \pi^{-2d} \prod_{i=1}^m \left( 2d_i+\frac{n_i}{2} \right)^{-d_i}  \leq \abs{det(T)}^{\frac{1}{\dim(P_{N,D})}} \leq \pi^{-2d} \prod_{i=1}^m \left( 1+\frac{n_i}{2d_i} \right)^{-d_i} $$

\end{lem}

\begin{proof}
Since the eigenspaces of $T$ are $r^{D-\alpha}H_{N,\alpha}$, we can explicitly write the determinant.

$$ \det(T)^{\frac{1}{\dim(P_{N,D})}} =  \prod_{\alpha \in \mathcal{I}(D) }  \left( \prod_{i=1}^m  \pi^{-\alpha_i} \frac{d_i! \Gamma(d_i+\frac{n_i}{2})}{(d_i-\frac{\alpha_i}{2})! \Gamma(d_i+\frac{\alpha_i+n_i}{2})} \right)^{\frac{\dim(H_{N,\alpha})}{\dim(P_{N,D})}} $$

For any  $1 \leq j \leq d_i$, we have 

$$ \left( \frac{1}{2d_i+\frac{n_i}{2}} \right)^{d_i} \leq \left( \frac{d_i-j}{d_i+\frac{n_i}{2}} \right)^{j} \leq \frac{d_i! \Gamma(d_i+\frac{n_i}{2})}{(d_i-j)! \Gamma(d_i+j+\frac{n_i}{2})}  \leq  \left( \frac{d_i}{d_i+j+\frac{n_i}{2}} \right)^{j} \leq \left( \frac{d_i}{d_i + \frac{n_i}{2}} \right)^{d_i}$$

Using this very rough estimates we complete the proof. 
\end{proof}

In previous section we introduced zonal harmonics and used them in several proofs. Analogs of zonal harmonics and ultraspherical polynomial do exist in multihomogenous setup. However, the following basic lemma seems to suffice for our purposes. So, we don't introduce zonal harmonics explicitly even though they are lurking in the background of our proofs. 

\begin{lem} \label{zonal}
Let $\{ y_i \}_{i=1}^{\dim(P_{N,D})}$ be an orthonormal basis of $P_{N,D}$ with respect to $L_2$ inner product. For all $v \in S$, we define a corresponding polynomial $q_v \in P_{N,D}$ as follows:

$$ q_v(x) := \sum_{i=1}^{\dim(P_{N,D})} y_i(v) y_i(x) $$

\noindent Then, $q_v$ satisfy the following properties.

\begin{enumerate}
\item For all $f \in P_{N,D}$, we have
   $$ \langle f ,q_v \rangle = f(v)  $$

\item For all $T \in SO(N)$ and $v \in S$, we have

$$ T \circ q_v = q_{Tv} $$

\item For all $v \in S$, we have

$$ q_v(v)=\norm{q_v}_2^2=\max_{x \in S} q_v(x)=\dim(P_{N,D}) $$

\item For any $f \in P_{N,D}$, we have  
            $$ \frac{\max_{x \in S} \abs{f(x)}}{\norm{f}_2} \leq \sqrt{\dim(P_{N,D})} $$
\noindent and the equality is satisfied if and only if $f=cq_v$ for some $v \in S$, and a constant $c$.						
\end{enumerate}
\end{lem}
\begin{proof}
For any $f \in P_{N,D}$ and for all $v \in S$, we have 

$$ f(v) = \sum_{i} \langle f , y_i \rangle y_i(v)  $$

\noindent This proves the first claim. For $T \in O(N)$ and $f \in P_{N,D}$, 

$$  \langle f , T \circ q_{v} \rangle  =  \int_{S} f(x) q_{v}(T^{-1} x) \; \sigma(x) = \langle T^{-1} \circ f , q_v \rangle = f(Tv)  $$

\noindent $f(Tv)=\langle f , q_{Tv} \rangle$ by the first property, and  $f \in P_{N,D}$ is arbitrary. This shows $T \circ q_{v} = q_{Tv}$. As a side result this also shows $q_v(v) = q_{Tv}(Tv)$ for all $v \in S$ and for all $T \in SO(N)$. Since $S$ is the $SO(N)$ orbit of any vector $v \in S$, we have the following consequence

$$  q_v(v) = \int_{S} q_{w}(w) \; \sigma(w) =  \int_{S} \sum_i y_i(w)^2 \; \sigma(w) = \sum_{i} \int_{S} y_i(w)^2 \; \sigma(w) = \dim(P_{N,D})$$

\noindent Again using the first property, we have $ \langle q_v , q_v \rangle = q_v(v)$ and $\abs{q_v(w)}=\abs{\langle q_v ,q_w \rangle} \leq \norm{q_v}_2 \norm{q_w}_2$. This completes the proof the third claim. Now for any $f \in P_{N,D}$ and for any $x \in S$, we have  $\abs{f(x)} \leq \norm{f}_2 \norm{q_x}_2 $. Since $\norm{q_x}_2 = \sqrt{\dim(P_{N,D})}$, we directly have 
  $$   \frac{\max_{x \in S} \abs{f(x)}}{\norm{f}_2} \leq \sqrt{\dim(P_{N,D})} $$
\noindent The equality case is given by the equality criterion of Cauchy-Schwarz inequality.
\end{proof}

\subsection{Barvinok's Inequality}

In this section we will present Barvinok's inequality for multihomogenous polynomials \cite{Bar02}. Barvinok's inequality is proved in the general setting of compact group orbits, where we only need the special case of the group $SO(N)$  acting on the vector space $V=(\mathbb{R}^{n_1})^{\otimes 2d_1} \times (\mathbb{R}^{n_2})^{\otimes 2d_2} \times \ldots \times (\mathbb{R}^{n_m})^{\otimes 2d_m}$.

To present Barvinok's inequality,  we need to introduce a bit of terminology. For $f \in P_{N,D}$, we define the $L_{\infty}$ norm of $f$ as follows:

$$ \norm{f}_{\infty} := \max_{v \in S}   \abs{f(v)} $$

We also need $L_{2k}$ norms for even integers $2k$.

$$ \norm{f}_{2k} := \left( \int_{S} f(v)^{2k} \; \sigma(v) \right)^{\frac{1}{2k}} $$

\begin{thm}[Barvinok's Inequality for Multihomogenous Forms] \label{Barvinok}
Let $f \in P_{N,D}$, let $k \in \mathbb{N}$, and set $d_{2k}= \prod_{i=1}^m \binom{2kd_i + n_i -1 }{2kd_i}$. Then, we have

$$  \norm{f}_{2k} \leq \norm{f}_{\infty} \leq d_{2k}^{\frac{1}{2k}} \norm{f}_{2k}  $$
\end{thm}

\section{The Cone of Nonnegative Polynomials}

We begin with recalling the definition of function $\mu$ from the introduction. We defined 

$$ C_{N,D}:=\left\{p \in P_{N,D} \; | \; \int_S p \: d\sigma=1\right\} $$ 

\noindent Then for any $X\!\subseteq\! P_{N,K}$  we set 

$$\mu(X)= \left( \frac{vol(X \cap C_{N,D})}{vol(B)} \right)^{\frac{1}{\dim(P_{N,D})}} $$

\noindent where $B$ is the unit ball ball with respect to $L_2$ inner product.

In this section we construct an isotropic measure introduced by the pointwise evaluation polynomials in Lemma \ref{zonal},  and show that convex hull of this isotropic measure is dual to the cone of nonnegative polynomials. Our upper bound for $\mu(\pos_{N,D})$ then follows from Theorem \ref{LYZ} via duality.

\begin{lem} \label{nice}
$$\mu(\pos_{N,D}) \leq 5$$
\end{lem}

\begin{proof}
We define a map $\Phi : S \rightarrow P_{N,D}$ by  setting 

$$ \Phi(v)  :=  \frac{q_v-r}{\sqrt{\dim(P_{N,D})-1}} $$

\noindent where $q_v$ is the polynomial corresponding to the vector $v$ as in Lemma \ref{zonal}, and $r=r_1^{d_1}r_2^{d_2} \ldots r_{m}^{d_m}$ with $r_1=x_1^2+x_2^2+\ldots+x_{n_1}^2$ and $r_i$ are all defined similarly. 

It is not hard to prove that $\Phi$ is Lipschitz and injective.  Now let $U \subseteq P_{N,D}$ be defined as follows. 

$$ U:=\{ p \in P_{N,D} : \langle p , r \rangle= 0 \} $$ 

\noindent For all $v \in S$, we have  $r(v)=\langle r , q_v \rangle=1$. This implies $Im(\Phi) \subseteq U$, and it also shows that $\norm{\Phi(v)}_2=1$ since $\norm{q_v}_2=\sqrt{\dim(P_{N,D})}$ for all $v \in S$. Let $\sigma_i$ be the uniform measure on $S^{n_i-1}$, and let $\sigma$ be the product of $\sigma_i$ as defined before.  Now, we define a measure $Z$ on the unit sphere of $U$, as the push-forward measure of  $\sigma$ under the map $\Phi$. It follows directly that $\supp(Z)=\mathrm{Image}(\Phi)$, and since $Z$ is a push-forward measure, it satisfies the following property for every function $g$ on U.

$$ \Int_{U} g(u) \: Z(u) = \Int_{S} g(\Phi(v)) \; \sigma(v) $$ 

\noindent Now observe that for every $f \in U$, we have the following equality:

$$ \norm{f}_2^2= \Int_{S} f(v)^2 \; \sigma(v) = \Int_{S} M \langle f, \Phi(v) \rangle^2 \; \sigma(v)= \Int_{U} M \langle f,u \rangle^2 \; Z(u) $$

\noindent where $M=\dim(U)=\dim(P_{N,D})-1$. This equality shows that $Z$ is an isotropic measure supported on the unit sphere of the vector space $U$!

To compute the centroid of $Z$ we consider the following polynomial $q$:

$$ q: =\Int_{S} q_v \; \sigma(v) $$

\noindent We observe that $q$ is invariant under the action of  $SO(N)$. This immediately yields that $q=ar$ for some constant $a$.  Since $\norm{r}_2=1$, and $\langle q_v ,r \rangle=1$, we have $q_v=r$. Thus $\frac{q-r}{\sqrt{M}}$ - which is the centroid of $Z$ - is the origin. Now using Theorem \ref{LYZ}, we deduce

$$\vol{{\conv(Im(\Phi))}^{\circ}} \leq \frac{M^{\frac{M}{2}} (M+1)^{\frac{M+1}{2}}}{M!}$$

\noindent where $M=\dim(U)=\dim(P_{N,D})-1$. Now we define the following convex body which will be useful in the rest of the article: 

$$ V:=\conv(\{ q_v-r : v \in S  \}) $$ 

\noindent where $q_v$ is the pointwise evaluation polynomial corresponding to the vector $v$ as defined in Lemma \ref{zonal}. Note that $V=\sqrt{M} \conv(\mathrm{Image}(\Phi))$. Using the inequality above, we have 

\begin{gather} \label{hmmm}
\abs{V^{\circ}} \leq \frac{ M^{\frac{M}{2}} (M+1)^{\frac{M+1}{2}} }{M! (\sqrt{M})^M} 
\end{gather}

\noindent Suppose $f \in C_{N,D}$ is given (i.e., $f \in P_{N,D}$ and $\langle f ,r \rangle=1$). Then, 

$$ f(v) \geq 0 \; \text{for all} \: v \in S \Leftrightarrow (f-r)(v) \geq -1  \Leftrightarrow \langle r-f , q_v-r \rangle \leq 1 $$

\noindent That is, $f \in C_{N,D} \cap \pos_{N,D}$ if and only if $-f+r \in V^{\circ}$. 

$$ C_{N,D} \cap \pos_{N,D} =  - V^{\circ} + r $$

\noindent Hence by (\ref{hmmm}), we have

$$ \mu(Pos_{N,D}) \leq (\frac{M^{\frac{M}{2}} (M+1)^{\frac{M+1}{2}}}{M! M^{\frac{M}{2}} \abs{B}})^{\frac{1}{M}} \leq \frac{\abs{B}^{\frac{-1}{M}} M^{\frac{1}{2}}}{\frac{M}{e}}$$

$$ \mu(Pos_{N,D}) \leq \frac{e}{\sqrt{M} \abs{B}^{\frac{1}{M}}} \leq 5 $$  

\end{proof}

\begin{rem}
Barvinok and Blekherman \cite{BarBle03} used John's Theorem to approximate volume of convex hulls of compact  group orbits. John's Theorem provides  very good approximation  for ellipsoid-like bodies but may not be sharp for convex bodies that do not resemble ellipsoids (i.e say bodies with large Banach-Mazur distance to the Euclidean unit ball). For instance, as far as we are able to compute Barvinok and Blekherman's Theorem yields an upper bound of the order $\sqrt{\dim(P_{N,D})}$ for $\mu(P_{N,D})$. 
\end{rem}

The following lemma states our lower bound for $\mu(\pos_{N,D})$. The construction carried out in the proof of Theorem \ref{nice} seems to indicate a lower bound via discretization and Vaaler's Inequality \cite{Vaaler}. For now we give a lower bound by using standard techniques combined with Barvinok's inequality (Theorem \ref{Barvinok}).

\begin{lem} \label{loose}
$$ \mu(\pos_{N,D}) \geq \frac{1}{ 4 \sqrt{\max_{i} \{ n_i \}} \prod_{i=1}^m \sqrt{2d_i+1} } $$
\end{lem}

\begin{proof}
Let's agree to call $\pos_{N,D} \cap C_{N,D}$ as $K$. Then $\mu(\pos_{N,D})=\left( \frac{\abs{K}}{\abs{B}} \right)^{\frac{1}{M}}$ where $M=\dim(P_{N,D})-1$. We will estimate the volume of $K-r$ from below. We defined a useful subspace in the previous proof: $ U := \{ f \in P_{N,D} :  \langle f , r \rangle = 0 \} $. 

Clearly $K -r \subseteq U$. Moreover, for $f \in U$, $f \in K-r$ if and only if $\min_{x \in S} f(x) \geq -1$. Using the Gauge function language introduced in the background section, we observe that $G_{K-r}(f)=\abs{\min_{x \in S} f(x) }$. Now we can express the volume of $K-r$ with $G_{K-r}$ using Lemma \ref{standard}.

$$ \left( \frac{\abs{K-r}}{\abs{B}} \right)^{\frac{1}{M}} = ( \int_{S^{M-1}} G_{K-r}(f)^{-M} \; \sigma_{M} (f) )^{\frac{1}{M}}  $$

\noindent where $S^{M-1}$ is the unit sphere of $U$, and $\sigma_M$ is the uniform measure on $S^{M-1}$. Clearly $G_{K-r}(f)=\abs{\min_{x \in S} f(x) } \leq \norm{f}_{\infty}$ for all $f$. So, we have

$$ \left( \frac{\abs{K-r}}{\abs{B}} \right)^{\frac{1}{M}}  \geq \left( \int_{S^{M-1}} \norm{f}_{\infty}^{-M} \; \sigma_M(f) \right)^{\frac{1}{M}} $$

Using H\"older and Jensen's inequalities, we have

$$ \left( \int_{S^{M-1}} \norm{f}_{\infty}^{-M} \; \sigma_M(f) \right)^{\frac{1}{M}} \geq \int_{S^{M-1}} \norm{f}_{\infty}^{-1} \; \sigma_M(f)  \geq \left( \int_{S^{M-1}} \norm{f}_{\infty} \; \sigma_M(f)  \right)^{-1} $$

\noindent So, it suffices to prove an upper bound for $\int_{S^{M-1}} \norm{f}_{\infty} \; \sigma_M(f)  $. We will use Barvinok's inequality for this purpose. Let $2k \geq 2$ be an even integer to be optimized later, and let $d_k=\prod_{i=1}^m \binom{2kd_i + n_i -1 }{2kd_i}$. Barvinok's inequality gives the following:

$$  \int_{S^{M-1}} \norm{f}_{\infty} \; \sigma_M(f) \leq d_k^{\frac{1}{2k}} \int_{S^{M-1}} \left( \int_{S} f(v)^{2k} \; \sigma(v) \right)^{\frac{1}{2k}} \sigma_{M}(f)   $$

\noindent Using H\"older's inequality and Fubini's theorem we have 

$$ \int_{S^{M-1}} \norm{f}_{\infty} \; \sigma_{M}(f) \leq d_k^{\frac{1}{2k}} \left(\int_{S}  \int_{S^{M-1}} f(v)^{2k} \; \sigma_{M}(f) \; \sigma(v) \right)^{\frac{1}{2k}} $$

\noindent The average inside the integral is independent of vector $v$, so we just need to compute the integral for any fixed $v$.  

$$  \Int_{S^{M-1}} f(v)^{2k} \; \sigma_{M}(f) = \int_{S^{M-1}} \langle f ,q_v - r \rangle^{2k} \; \sigma_{M}(f)  $$

\noindent Note that we know $\norm{q_v-r}_2=\sqrt{M}$. So, we obtain 

$$ \left( d_k \int_{S^{M-1}} \langle f ,q_v \rangle^{2k} \; \sigma_{M}(f) \right)^{\frac{1}{2k}} \leq {d_k}^{\frac{1}{2k}} \sqrt{M} \left( \frac{\Gamma(k+\frac{1}{2}) \Gamma(\frac{1}{2}M)}{ \pi^{\frac{M}{2}}\Gamma(\frac{1}{2}M+k)} \right)^{\frac{1}{2k}} $$ 

\noindent If $k \leq M$, we have

$$ \left(\frac{\Gamma(k+\frac{1}{2}) }{\pi^{\frac{M}{2}}}\right)^{\frac{1}{2k}} \leq \sqrt{k}  \; \;   \text{and} \; \; \left(  \frac{\Gamma(\frac{1}{2}M)}{\Gamma(\frac{1}{2}M+k)} \right)^{\frac{1}{2k}} \leq  \sqrt{\frac{2}{M}} $$

$$ \Int_{S^{M-1}} \lVert f \rVert_{\infty} \: df \leq d_k^{\frac{1}{2k}} \sqrt{M} \sqrt{k} \sqrt{\frac{2}{M}} \leq d_k^{\frac{1}{2k}} \sqrt{2k} $$

\noindent Now we need to pick $k \leq M$ in an optimal way to minimize  $d_k^{\frac{1}{2k}}$. We set  $h=\max\{ n_i \}$,  and set $k= h \prod_i (2d_i+1)$. Note that we always have $k \geq 3^m h$, $m \geq 2$.

$$ d_k^{\frac{1}{2k}} = \prod_{i=1}^m \binom{2kd_i + n_i -1 }{2kd_i}^{\frac{1}{2k}} \leq \prod_i \left( \frac{ek(2d_i+1)}{n_i} \right)^{\frac{n_i}{2k}} \leq (\frac{ek^2}{h})^{\frac{mh}{2k}} \leq (9\sqrt{e})^{\frac{2}{9}}  < 2\sqrt{2}$$

\noindent Here we used $ (\frac{k}{n_i})^{\frac{n_i}{k}} \leq (\frac{k}{h})^{\frac{h}{k}}$. Hence, we have proved the following. 

$$ \Int_{S^{M-1}} \lVert f \rVert_{\infty} \; \sigma_{M}(f) \leq  4 \sqrt{ \max \{ n_i \} } \prod_{i} \sqrt{2d_i + 1} $$
\end{proof}

\section{The Cone of Sums of Squares}
In this section we prove our bounds for $\mu(\sq_{N,D})$. We  start with the upper bound.

\begin{lem} \label{squp}

$$  \mu(\sq_{N,D})  \leq 4 \prod_i 2^{5d_i} e^{\frac{d_i}{2}} (\frac{d_i}{n_i+2d_i})^{\frac{d_i}{2}} $$

\end{lem}

\begin{proof}
We call $\sq_{N,D} \cap C_{N,D}$ as $K$, and we work on $K-r$. Let's recall the definition of mean width.

$$ \omega(K-r) := \int_{S^{M-1}} h_{K-r}(f) \; \sigma_M(f) = \int_{S^{M-1}}  \max_{g \in K-r} \langle f ,g \rangle \; \sigma_{M}(f)$$ 

\noindent where $S^{M-1}$ is the unit sphere of the subspace $U := \{ f \in P_{N,D} : \langle f ,r \rangle= 0 \}$, and $\sigma_M$ is the uniform measure on it. By Theorem \ref{urysohn} (Urysohn's inequality),  we have   

$$\mu(\sq_{N,D})  \leq \omega(K-r) $$ 

\noindent For a $g \in P_{N,\frac{D}{2}}$ with $g^2 \in \sq_{N,D} \cap C_{N,D}$, we have by definition $\int_S g^2(v) \; \sigma(v) = 1$. Thus $\norm{g}_2=1$. So, all extreme points of $K-r$ are of the form $g^2 - r$ for a $g \in P_{N,\frac{D}{2}}$ with $\norm{g}_2=1$. For $f \in U$, we also have $ \langle f , g^2-r \rangle= \langle f , g^2 \rangle$. So, we write 

$$ h_{K-r}(f) \leq  \max_{g \in P_{N,D/2} , \lVert g \rVert=1} \langle f,g^2 \rangle $$ 

\noindent This gives us an easy inequality for the mean width.

$$ \omega(K-r) \leq \int_{S^{M-1}}  \max_{g \in P_{N,D/2} , \lVert g \rVert=1} \abs{ \langle f,g^2 \rangle}   \; \sigma_{M}(f) $$

\noindent For a fixed $f$, $\langle f,g^2 \rangle$ is a quadratic form on variable $g$. Let's call this quadratic form $Q(f)$, then we have  $\norm{Q(f)}_{\infty} = \max_{g \in P_{N,D/2} , \lVert g \rVert=1} \abs{ \langle f,g^2 \rangle}$. We apply Barvinok's inequality to $Q(f)$ with exponent $k$ to be optimized later.

$$  \omega(K-r) \leq  d_{k}^{\frac{1}{2k}} \int_{S^{M-1}} \left( \int_{S^{D-1}} \langle f,g^2 \rangle^{2k} \; \sigma_D(g) \right)^{\frac{1}{2k}} \sigma_{M}(f) $$

\noindent where $S^{D-1}$ is the unit sphere of $P_{N,\frac{D}{2}}$. Using some help from H\"older and Fubini,  we have

$$  \omega(K-r) \leq  d_{k}^{\frac{1}{2k}} \left( \int_{S^{D-1}}  \int_{S^{M-1}} \langle f,g^2 \rangle^{2k} \; \sigma_{M}(f) \; \sigma_D(g) \right)^{\frac{1}{2k}} $$

\noindent  Thanks to Reverse H\"older inequalities of J.\ Duoandikoetxea \cite{Duo87}, we have $ \norm{g^2}_2 \leq 4^{deg(g^2)} $. Note that $deg(g^2)=\sum_i 2d_i=2d$. Now we are in the same situation as in the last part of the proof of Lemma \ref{loose}. So, we have

$$ \left( d_k \int_{S^{M-1}} \langle f , g^2 \rangle^{2k} \; \sigma_{M}(f) \right)^{\frac{1}{2k}} \leq {d_k}^{\frac{1}{2k}}   4^{2d} \left( \frac{\Gamma(k+\frac{1}{2}) \Gamma(\frac{1}{2}M)}{ \pi^{\frac{M}{2}}\Gamma(\frac{1}{2}M+k)} \right)^{\frac{1}{2k}} $$ 

\noindent For $k \leq M$, we have

$$  \left( d_k \int_{S^{M-1}} \langle f , g^2 \rangle^{2k} \; \sigma_{M}(f) \right)^{\frac{1}{2k}} \leq {d_k}^{\frac{1}{2k}}  4^{2d} \sqrt{\frac{2k}{M}}  $$

\noindent $Q(f)$ is a quadratic form in $\dim(P_{N,\frac{D}{2}})$ many variables, so $d_k=\binom{2k + \dim(P_{N,\frac{D}{2}})-1}{2k}$. We set $2k=\dim(P_{N,\frac{D}{2}})$. This gives $d_k^{\frac{1}{2k}} \leq 4$. So, we have proved the following.

$$ \mu(\sq_{N,D})  \leq \omega(K-r) \leq 4^{2d+1} \left( \frac{\dim(P_{N,\frac{D}{2}})}{\dim(P_{N,D})-1} \right)^{\frac{1}{2}} $$ 

\noindent where $\frac{D}{2}=(d_1,d_2,\ldots,d_m)$ and $d=d_1+d_2+\ldots+d_m$. Stirling's estimate gives us the following:

$$ \frac{\dim(P_{N,\frac{D}{2}})}{\dim(P_{N,D})} \leq \frac{\prod_i (\frac{e(n_i+d_i)}{d_i} )^{d_i} }{ \prod_i ( \frac{n_i+2d_i}{2d_i} )^{2d_i}  } \leq \prod_i 2^{2d_i}e^{d_i} (\frac{d_i}{n_i+2d_i})^{d_i}  $$

$$ \mu(\sq_{N,D})  \leq \omega(K-r) \leq 4 \prod_i 2^{5d_i} e^{\frac{d_i}{2}} (\frac{d_i}{n_i+2d_i})^{\frac{d_i}{2}} $$
\end{proof}

To prove the lower bound for $\mu(\sq_{N,D})$ we need the following lemma which was essentially proved by Blekherman as Lemma 5.3 at \cite{Ble06} 

\begin{lem}[Blekherman's observation]
$$ \sq_{N,D}^{d*} \subseteq \sq_{N,D} $$
where $ \sq_{N,D}^{d*}$ is the dual cone with respect to the differential metric. 
\end{lem}

\begin{lem}
 $$ \mu(\sq_{N,D}) \geq  c_s  \prod_{i=1}^m (2^{10} \pi^4)^{-\frac{d_i}{2}} \left( \frac{n_i}{2} + 2d_i\right)^{-\frac{d_i}{2}}  $$
\noindent where $0<c<1$ is a universal constant.
\end{lem}

\begin{proof}
We set $K=\sq_{N,D} \cap C_{N,D}$, and  $K_1=\sq_{N,D}^{d*} \cap C_{N,D}$. Due to Blekherman's observation, we have $K_1 \subseteq K$. We also have that for $f \in K_1 - r$ and $g \in K -r$, 

$$ \langle f - r , g -r \rangle_D = \langle f , g \rangle_D - \langle r , r \rangle_D  \geq - A^{-1} \prod_i (2d_i)!  .$$

\noindent Also note that this inequality is reversible and it is a description of the convex body $K_1-r$.

Now we set $K_2=\sq_{N,D}^{\circ} \cap C_{N,D}$, where $\sq_{N,D}^{\circ}$ denotes the dual cone with respect to $L_2$ inner product. As in our preceding proofs, we have $K_2-r=-(K-r)^{\circ}$, where $(K-r)^{\circ}$ is the polar body with respect to $L_2$-inner product. Now, let $f \in T(K_2)$ with $f=T(h)$, and let $g \in K$. Then,

$$ \langle f -r , g -r \rangle_D = A^{-1} \prod_{i} (2d_i)! \langle h -r , g -r \rangle \geq - A^{-1} \prod_{i} (2d_i)!  $$

\noindent where we used $r=T(r)$. This inequality shows that $T(K_2) - r \subseteq K_1 -r$.  Combining this inclusion with the fact that $K_1 \subseteq K$ gives us the following inequality.

$$  \det(T)^{\frac{1}{M}} \left( \frac{\abs{K_2-r}}{\abs{B}} \right)^{\frac{1}{M}} \leq \left( \frac{\abs{K_1-r}}{\abs{B}} \right)^{\frac{1}{M}} \leq  \left( \frac{\abs{K}}{\abs{B}} \right)^{\frac{1}{M}} $$

\noindent where $M$ is the dimension of the vector space $U := \{ q \in P_{N,D} : \langle q ,r \rangle = 0 \}$. Note that Reverse Santalo inequality (Theorem \ref{bigboe}) gives us the following:

$$  c_s \left( \frac{\abs{B}}{\abs{K}}  \right)^{\frac{1}{M}}  \leq  \left( \frac{\abs{K_2-r}}{\abs{B}} \right)^{\frac{1}{M}} $$

\noindent So, we have   

$$  c_s \det(T)^{\frac{1}{M}} \left( \frac{\abs{B}}{\abs{K}}  \right)^{\frac{1}{M}}    \leq  \left( \frac{\abs{K}}{\abs{B}} \right)^{\frac{1}{M}} $$

\noindent Combining the upper bound in Lemma \ref{squp}, and the lower bound for $det(T)^{\frac{1}{M}}$ in Lemma \ref{detT} gives the following.

$$  \left( \frac{\abs{K}}{\abs{B}} \right)^{\frac{1}{M}} \geq c \pi^{-2d}  \prod_{i=1}^m (2d_i + \frac{n_i}{2})^{-d_i} \left( \frac{n_i+ 2d_i}{2^{10}e d_i}\right)^{\frac{d_i}{2}} \geq c_s (2^{10} \pi^4)^{-\frac{d}{2}} \prod_{i=1}^m \left( \frac{n_i}{2} + 2d_i\right)^{-\frac{d_i}{2}}  $$
 
\end{proof}

\section{The Cone of Powers of Linear Forms}

This section develops quantitative bounds on the cone of even powers of linear forms. 

$$ L_{N,D}:=\{ p \in \pos_{N,D} : p=\sum_{i} l_{i1}^{2d_1} l_{i2}^{2d_2} \cdots l_{im}^{2d_m} \; \text{where} \; l_{ij} \; \text{are linear forms in} \; \ox_j \}  $$

We will consider volume bounds for the following section of $L_{N,D}$:

$$ L_{N,D} \cap C_{N,D} = \{ f \in L_{N,D} : \langle f , r \rangle = 1 \} = \{ f \in L_{N,D} : \langle f , r \rangle_D = A^{-1} (2d_1)! (2d_2)! \ldots (2d_m)! \} $$ 

\noindent For convenience in the rest of this section we set $K:=L_{N,D} \cap C_{N,D} $ and $\lambda:=\prod_{i}(2d_i)!$. 

Extreme points of $K$ are of the form $l_{1}^{2d_1} l_{2}^{2d_2} \cdots l_{m}^{2d_m}$ for some linear forms $l_i=\langle c_i , \ox_i \rangle$. One can scale all these $c_i$, and write $l_{1}^{2d_1} l_{2}^{2d_2} \cdots l_{m}^{2d_m} = \prod_i \norm{c_i}_2^{2d_i} \langle x , \frac{c_i}{\norm{c}_2} \rangle^{2d_i}$. Note that in this case, we have

$$  A^{-1} \lambda = \langle r, \prod_i \norm{c_i}_2^{2d_i} \langle \ox_i , \frac{c_i}{\norm{c}_2} \rangle^{2d_i} \rangle_D = \lambda \prod_{i} \norm{c_i}_2^{2d_i} $$

\noindent This shows that extreme points of $K$ are of the form $A^{-1} \prod_{i} \langle \ov_i , \ox_i \rangle^{2d_i}$ for some $v \in S$. 

Recall that for $v \in S$, we have already defined $\delta_v=\prod_{i} \langle \ov_i , \ox_i \rangle^{2d_i}$ in the tale of inner products section. Our discussion so far, combined with Krein-Milman theorem gives the following result. 

$$ K = conv \{ A^{-1} \delta_v : v \in S \}$$
  
Recall that we also showed existence of polynomials $q_v$ in Lemma \ref{zonal}, with the property that $\langle f , q_v \rangle=f(v)$ for all $f \in P_{N,D}$. Now, let $T$ be the operator as defined in the tale of inner products section, and consider $T(q_v)$:

$$ \langle f , T(q_v) \rangle_D = A^{-1} \lambda \langle f , q_v \rangle = A^{-1} \lambda f(v) $$

\noindent Since $f$ is arbitrary, and $\langle f ,\delta_v \rangle_D = \lambda f(v)$ this shows that $T(q_v)=A^{-1} \delta_v$ for all $v \in S$. This gives another description of $K$.

$$ K = conv \{ T(q_v)  : v \in S \} $$

\noindent In the section on volume bounds for the cone of nonnegative polynomials, we have defined the convex body $ V := \{ q_v-r : v \in S \}$, and we proved volume bounds for the polar of $V$:

$$  \frac{1}{5}  \leq (\frac{\abs{B}}{\abs{V^{\circ}}})^{\frac{1}{\dim(P_{N,D})}} \leq   4 \sqrt{\max_{i} \{ n_i \}} \prod_{i=1}^m \sqrt{2d_i+1} $$

\noindent This translates to the following lower bound for $\abs{V}$ via reverse Santalo inequality (Theorem \ref{bigboe}).

$$  \frac{c_s}{5}   \leq (\frac{\abs{V}}{\abs{B}})^{\frac{1}{\dim(P_{N,D})}}  $$

\noindent For an upper bound on $(\frac{\abs{V}}{\abs{B}})^{\frac{1}{\dim(P_{N,D})}} $, we would like to use Santalo inequality (Theorem \ref{Santalo}). So, we need to find out the Santalo point of $V$. Note that the Santalo point of $V$ is unique, and $V$ is invariant under $SO(N)$ action. Also note that the unique point in $V$, which is invariant under $SO(N)$ action is the origin, and hence it is the Santalo point of $V$. So, this gives us the following upper bound.

$$ (\frac{\abs{V}}{\abs{B}})^{\frac{1}{\dim(P_{N,D})}} \leq (\frac{\abs{B}}{\abs{V^{\circ}}})^{\frac{1}{\dim(P_{N,D})}} \leq 4 \sqrt{\max_{i} \{ n_i \}} \prod_{i=1}^m \sqrt{2d_i+1} $$

\noindent We observed that $K=T(V)$, and we already had some upper and lower bound for $det(T)$ in Lemma \ref{detT}.  Also recall that  $ \mu(L_{N,D}) = \left( \frac{\abs{K}}{\abs{B}} \right)^{\frac{1}{\dim(P_{N,D})}}$. All together, these facts give us the following result.
\begin{thm}
$$ c_s  \pi^{-2d} \prod_{i=1}^m \left( 2d_i+\frac{n_i}{2} \right)^{-d_i}  \leq  \mu(L_{N,D}) \leq 4 \pi^{-2d} \sqrt{\max_{i} \{ n_i \}}  \prod_{i=1}^m \sqrt{2d_i+1} \left( 1+\frac{n_i}{2d_i} \right)^{-d_i}$$
\end{thm}

\section{Acknowledgements} 
I would like to thank Greg Blekherman for useful discussions over e-mail. Ideas developed in Greg Blekherman's articles had a strong influence on parts of this note. I also would like to thank Petros Valettas and Grigoris Paouris for helpful discussions and splendid Greek hospitality at Athens, College Station and wherever else we were able to meet. While I was writing this note, I was enjoying hospitality of \"Ozgur Ki\c{s}isel  at METU, many thanks go to him. Last but not the least, I would like to thank J. Maurice Rojas for introducing me to quantitative aspects of Hilbert's \nth{17} problem, and for many useful discussions.

\end{document}